 \documentclass[final,1p,times]{elsarticle}

\usepackage{amssymb}
 \usepackage{amsthm}
\usepackage{amsmath,amssymb,amsopn,amsfonts,mathrsfs,amsbsy,amscd}
\usepackage{longtable}
 \usepackage{lineno}

\newcommand{\prs}{\langle\;,\;\rangle}

\newcommand{\too}{\longrightarrow}
\newcommand{\om}{\omega}

\newcommand{\esp}{\quad\mbox{and}\quad}

\def\br{[\;,\;]}

\newcommand{\G}{\mathfrak{g}}
\newcommand{\g}{\mathfrak{g}}
\newcommand{\rf}{\mathfrak{r}}
\newcommand{\uf}{\mathfrak{u}}
\newcommand{\af}{\mathfrak{a}}

\newcommand{\h}{{\mathfrak{h}}}

\newcommand{\s}{{\mathfrak{s}}}

\newcommand{\ad}{{\mathrm{ad}}}

\newcommand{\tr}{{\mathrm{tr}}}
\newcommand{\ric}{{\mathrm{ric}}}

\newcommand{\B}{{\cal B}}

\newcommand{\p}{{\mathfrak{p} }}
\newcommand{\n}{{\mathfrak{n} }}
\newcommand{\kf}{{\mathfrak{k} }}

\newcommand{\di}{\displaystyle}

\newcommand{\na}{\nabla}

\newcommand{\Ric}{\mathrm{Ric}}

\newcommand{\al}{\alpha}
\newcommand{\be}{\beta}

\newcommand{\Ga}{\Gamma}
\newcommand{\e}{\epsilon}

\newcommand{\la}{\lambda}

\newcommand{\de}{\delta}

\newtheorem{theo}{Theorem}[section]
\newtheorem{pr}{Proposition}[section]

\newtheorem{co}{Corollary}[section]

\newtheorem{conjecture}{Conjecture}
\font\bb=msbm10

\def\B{\hbox{\bb B}}
\def\R{\hbox{\bb R}}

\begin{document}

\begin{frontmatter}
	
	
	
	\title{ Left invariant Riemannian metrics with harmonic curvature are 
		Ricci-parallel in solvable Lie groups and Lie groups of dimension $\leq6$}
	
	\author[label1,label2]{Ilyes Aberaouze, Mohamed Boucetta}
	\address[label1]{Universit\'e Cadi-Ayyad\\
		Facult\'e des sciences et techniques\\
		BP 549 Marrakech Maroc\\e-mail: aberaouzeilyes@gmail.com 
		
	}
	\address[label2]{Universit\'e Cadi-Ayyad\\
		Facult\'e des sciences et techniques\\
		BP 549 Marrakech Maroc\\e-mail: m.boucetta@uca.ac.ma
	}
	
	
	
	
	
	\begin{abstract} We show that any left invariant metric with harmonic curvature on a solvable Lie group is Ricci-parallel. We show the same result for any Lie group of dimension $\leq$ 6.

	\end{abstract}
	
		
\end{frontmatter}

\section{Introduction}\label{section1}

A Riemannian manifold $(M,\prs)$ is said to have {\it harmonic curvature} if its curvature tensor $R$ has a vanishing divergence, i.e., for any $X\in\Ga(TM)$,
\[ \sum_{i=1}^n\na_{E_i}(R)(E_i,X)=0, \]where $\na$ is the Levi-Civita connection and $(E_1,\ldots,E_n)$ is a local $\prs$-orthonormal frame field.
The study of Riemannian manifolds with harmonic curvature is old and goes to \cite{Lich} and
it is known \cite{bour, der1, der2, gray} that a Riemannian manifold $(M,\prs)$ has a harmonic curvature if and only if its Ricci tensor $\Ric$ is  Codazzi, i.e.,
\begin{equation}\label{eq}
\na_X(\Ric)(Y,Z)=\na_Y(\Ric)(X,Z)
\end{equation}for any vector fields $X,Y,Z$. Obviously, any Einstein Riemannian metric $(\Ric=\la\prs)$ and more generally any Ricci-parallel metric $(\na(\Ric)=0)$ satisfy \eqref{eq}. Thus the  interesting case will be where $\Ric$ satisfies \eqref{eq} and $\na(\Ric)\not=0$. Unfortunately, such metrics are difficult to find and some examples were given in \cite{bess, der1, der2, gray}. To our knowledge, there is no example of a non Ricci-parallel homogeneous Riemannian manifold with  harmonic curvature  which supports the following conjecture.
\begin{conjecture}\label{conj} Any homogeneous Riemannian manifold $M$ with harmonic curvature is Ricci-parallel.
	
	\end{conjecture} 
This conjecture is true in dimension four (see \cite{spiro}) and when $M$ is a sphere or a projective space (see \cite{peng}). It was proven in \cite{atri} for nilpotent Lie groups with left invariant metrics. Moreover, any conformally flat Riemannian manifold with constant scalar curvature  satisfies \eqref{eq} (see \cite[Theorem 5.1]{gray}) and any homogeneous conformally flat Riemannian manifold has Ricci-parallel curvature (see \cite{hito}). The purpose of this paper is to show that this conjecture is true when $M$ is a solvable Lie group endowed with a left invariant metric  which implies that the conjecture is true for any homogeneous Riemannian manifold $G/K$ where $K$ is a maximal compact subgroup of $G$. We show also that the conjecture is true for any  Lie group of dimension $\leq 6$ endowed with a left invariant metric.

The paper is organized as follows. In Section \ref{section2}, we recall some basic properties of the Lie algebra associated to a Lie group endowed with a left invariant metric with harmonic curvature. In Section \ref{section3}, we prove  Conjecture \ref{conj} for a solvable Lie group. This is based on a generalization of the  main result of Lauret in \cite{lauret}. In Section \ref{section4}, we prove  Conjecture \ref{conj} for any Lie group of dimension $\leq6$. The proof for the dimensions $\leq5$ is quite easy but, for  dimension 6, it involves a huge computation where the use of Maple is needed.

\section{Preliminaries}\label{section2}
In this section, we establish the basic tools needed for the proof of our main results.
The results of this section were first proved in \cite{atri}.

The study of left invariant Riemannian metrics on Lie groups reduces to the study of  Lie algebras endowed with a Euclidean product.    Let $(\G,\br,\prs)$ be a Euclidean Lie algebra, i.e., a Lie algebra endowed with a bilinear nondegenerate form which is  positive definite. The \emph{ Levi-Civita product} of ${\G}$ is the bilinear map $\mathrm{L}:{\G}\times{\G}\too{\G}$ given by  Koszul's
formula
\begin{eqnarray}\label{levicivita}2\langle
\mathrm{L}_uv,w\rangle&=&\langle[u,v],w\rangle+\langle[w,u],v\rangle+
\langle[w,v],u\rangle \quad\mbox{for any}\; u,v,w\in\G.\end{eqnarray}
For any $u,v\in{\G}$, $\mathrm{L}_{u}:{\G}\too{\G}$ is skew-symmetric and $[u,v]=\mathrm{L}_{u}v-\mathrm{L}_{v}u$.
The curvature of ${\G}$ is given by
$$
\label{curvature}K(u,v)=\mathrm{L}_{[u,v]}-[\mathrm{L}_{u},\mathrm{L}_{v}] \quad\mbox{for any}\; u,v,\in\G.
$$
The Ricci curvature $\mathrm{ric}:{\G}\times{\G}\too\R$ and its Ricci operator $\Ric:{\G}\too{\G}$ are defined by $$\langle \Ric (u),v\rangle=\mathrm{ric}(u,v)=\mathrm{tr}\left(w\too
K(u,w)v\right) \quad\mbox{for any}\; u,v,w\in\G.$$ The operator $\Ric$ is symmetric and hence is diagonalizable. Thus there exists $\la_1<\ldots<\la_r$ such that 
\[ \G=\g_1\oplus\ldots\oplus\g_r \]where $\G_i$ is the eigenspace of $\Ric$ associated to the eigenvalues $\la_i$.

We call $(\G,\prs)$ a Euclidean Lie algebra with harmonic curvature if  $\Ric$ is a Codazzi operator, i.e., for any $u,v,w\in\G$,
\begin{equation}\label{codazzi}
\langle\mathrm{L}_uv,\Ric(w)\rangle+
\langle\mathrm{L}_uw,\Ric(v)\rangle=\langle\mathrm{L}_vu,\Ric(w)\rangle+
\langle\mathrm{L}_vw,\Ric(u)\rangle.
\end{equation}If we take $u\in\g_i$, $v\in\g_j$ and $w\in\g_k$,  we get
\begin{equation}\label{eq1}
(2\la_k-\la_i-\la_j)\langle [u,v],w\rangle=(\la_i-\la_j)(\langle[v,w],u\rangle+\langle[u,w],v\rangle).
\end{equation}
For $i=j\not=k$, we get that $\G_i$ is a subalgebra and if we take $k=i\not=j$, we get
\[ \langle [v,u],w\rangle+\langle[v,w],u\rangle=0\quad u,w\in\G_i, v\in\G_j, i\not=j. \]For any $i<j<k$ and $u_i\in\G_i,u_j\in\G_j,u_k\in\G_k$, we have
\begin{equation}\label{sy} \begin{cases}(2\la_i-\la_j-\la_k)\langle [u_j,u_k],u_i\rangle-(\la_j-\la_k)(\langle[u_k,u_i],u_j\rangle+
\langle[u_j,u_i],u_k\rangle)=0,\quad(E1)\\
(2\la_j-\la_i-\la_k)\langle [u_i,u_k],u_j\rangle-(\la_i-\la_k)(\langle[u_k,u_j],u_i\rangle+
\langle[u_i,u_j],u_k\rangle)=0,\quad(E2)\\
(2\la_k-\la_i-\la_j)\langle [u_i,u_j],u_k\rangle-(\la_i-\la_j)(\langle[u_j,u_k],u_i\rangle+
\langle[u_i,u_k],u_j\rangle)=0.\quad (E3)
\end{cases} \end{equation}We have
\[ (E1)-(E2)+(E3)=0 \]and if we put
\[ x=\langle[u_k,u_j],u_i\rangle,\; y=\langle[u_k,u_i],u_j\rangle\esp z=\langle[u_i,u_j],u_k\rangle, \]then the system \eqref{sy} becomes
\[ \begin{cases}(\la_j+\la_k-2\la_i)x+(\la_k-\la_j)y=(\la_k-\la_j)z,\\
(\la_k-\la_i)x+(\la_i+\la_k-2\la_j)y=(\la_i-\la_k)z.
\end{cases} \]So
\[ x=-{\frac { \left( \la_{{j}}-\la_{{k}} \right) ^{2}z}{ \left( \la_{
			{i}}-\la_{{j}} \right) ^{2}}}\esp y={\frac { \left( \la_{{i}}-\la_{{k}}
		\right) ^{2}z}{ \left( \la_{{i}}-\la_{{j}} \right) ^{2}}}. 
\]Thus
\[ (\la_i-\la_j)^2\langle[u_j,u_k],u_i\rangle+\left( \la_{{j}}-\la_{{k}} \right) ^{2}
\langle[u_j,u_i],u_k\rangle=0\esp
\left( \la_{{i}}-\la_{{j}} \right) ^{2}\langle[u_i,u_k],u_j\rangle+
\left( \la_{{i}}-\la_{{k}}\right) ^{2}\langle[u_i,u_j],u_k\rangle=0. \]
So far, we proved the following theorem.
\begin{theo}\label{prel} Let $(\G,\br,\prs)$ be a Euclidean Lie algebra, $\la_1<\ldots<\la_r$ the eigenvalues of $\;\Ric$ and  $\G=\G_1\oplus\ldots\oplus\G_r$ is the splitting of $\G$ on eigenspaces of $\Ric$. Then $\Ric$ is a Codazzi tensor if and only if:
	\begin{enumerate}\item for any $i\in\{1,\ldots,r\}$, $\G_i$ is a subalgebra of $\G$,
		\item for any $i,j\in \{1,\ldots,r\}$ with $i\not=j$ and for any $u,v\in\G_i$ and $w\in\G_j$ then
		\begin{equation}\label{sk} \langle [w,u],v\rangle+\langle [w,v],u\rangle=0, \end{equation}
		\item for any $i<j<k$ for any $u\in\G_i,v\in\G_j,w\in\G_k$ then
		\begin{equation}\label{3} \left(\la_i-\la_j\right)^2\langle[v,w],u\rangle+\left( \la_{j}-\la_{k} \right)^{2}
		\langle[v,u],w\rangle=0\esp
		\left( \la_{i}-\la_{j} \right)^{2}\langle[u,w],v\rangle+
		\left( \la_{i}-\la_{k}\right)^{2}\langle[u,v],w\rangle=0. \end{equation}
	\end{enumerate}In this case, $\Ric$ is non parallel if and only if there exists a triple $(i,j,k)$ with $i<j<k$ and $u_0\in\G_i$, $v_0\in\G_j$ and $w_0\in\G_k$ such that
	$\langle [u_0,v_0],w_0\rangle\not=0$. 
\end{theo}

Let us give an interpretation of  assertion 3. in the theorem. For any $u\in\G$, denote by $u_i$ its component in $\G_i$. Fix an index $k$ and define a new product $\prs_k$ on $\G$ by
\[ \langle u,v\rangle_k=\sum_{j=1}^r(\de_{jk}+(\la_j-\la_k)^2)\langle u_j,v_j\rangle,\quad u,v\in\G. \]Denote by $\G_k^\perp$ the orthogonal of $\G_k$.  It is easy to check that, for any $u_k\in\G_k$ and for any $u,v\in\G$,
\begin{equation}
\langle [u_k,v],w\rangle_k+\langle[u_k,w],v\rangle_k=0
\end{equation} and \begin{equation}\label{rho}\rho_k:\G_k\too \mathrm{so}(\G_k^\perp,\prs_k)\quad\mbox{where}\quad \rho_k(u_k)(v)=\sum_{l\not=k}[u_k,v]_l,\end{equation} is a representation. The proof of the following proposition is given  in \cite{atri}.
\begin{pr} Let $(\G,\br,\prs)$ be a Euclidean Lie algebra whose Ricci operator is Codazzi. With the notations above, for any $i=1,\ldots,r$, let $\{V_{i,\al}\}$ be an orthonormal basis of $\G_i$. Then, for any $i=1,\ldots,r$, $u_i\in\G_i$,
	\begin{equation}\label{r}
	\ric(u_i,u_i)=\ric_i(u_i,u_i)-\sum_{j\not=k}\sum_{\al,\be}
	\frac{(\la_k-\la_i)(\la_j-\la_i)}{(\la_k-\la_j)^2}
	\langle[V_{j,\al},V_{k,\be}],u_i\rangle^2,
	\end{equation}where $\ric_i$ is the Ricci curvature of $\G_i$ endowed with the restriction of $\prs$. Moreover, the scalar curvature $\s$ of $\G$ satisfies
	\begin{equation}\label{s}
	\s=\s_1+\ldots+\s_r,
	\end{equation}where $\s_i$ is the scalar curvature of $\G_i$.
	
	\end{pr}
	
	\section{Left invariant Riemannian metrics with harmonic curvature on solvable Lie groups are Ricci-parallel}\label{section3}
	
	The proof of our main result is based on the following theorem which is a generalization of the famous theorem of Lauret \cite{lauret} which shows that any solvable Einstein Euclidean Lie algebra is standard in the sense that the orthogonal of its derived ideal is abelian. Actually, the proof of Lauret's theorem is still valid if we suppose that $\Ric_{|[\G,\G]}=c\mathrm{Id}_{[\G,\G]}$ instead of $\Ric=c\mathrm{Id}_\G$. To convince the reader we copy Lauret's proof and we give the minor modifications needed to adapt to our hypothesis.
	
	\begin{theo}\label{l} Let $(\G,\br,\prs)$ be a solvable Lie algebra such that there exists a constant $c$ such that $\Ric_{|[\G,\G]}=c\mathrm{Id}_{[\G,\G]}$. Then $\G$ is standard, i.e., $[\G,\G]^\perp$ is abelian.
		
		\end{theo}
		
		\begin{proof} Put $\n=[\G,\G]$ and $\af=[\G,\G]^\perp$. We have
			\[ \Ric=R-\frac12B-S(\ad_H) \]where $S(\ad_H)=\frac12(\ad_H+\ad_H^*)$, $\langle H,u\rangle=\tr(\ad_u)$, $\langle Bu,v\rangle=\tr(\ad_u\circ\ad_v)$ and
			\begin{equation} \begin{cases}\di\langle Ru,v\rangle=-\frac12\sum_{i,j}\langle[u,e_i],e_j\rangle
			\langle[v,e_i],e_j\rangle+\frac14\sum_{i,j}\langle[e_i,e_j],u\rangle
			\langle[e_i,e_j],v\rangle, \\\di
			\tr(\Ric+\frac12B+S(\ad_H))E=\frac14\sum_{i,j}\langle E([e_i,e_j])-[E(e_i),e_j]-[e_i,E(e_j)],[e_i,e_j]\rangle,\end{cases}
			\label{eq4}\end{equation}for any endomorphism $E$ of $\G$ and where $(e_1,\ldots,e_n)$ is an orthonormal basis of $\G$. Since $\Ric$ leaves $\n$ invariant it leaves $\af$ invariant  and
			for any, $s\in\af$, 
			\begin{align}\label{k}
			\ric(s,s)&=-\frac12\sum_{i,j}\langle[s,e_i],e_j\rangle^2-\frac12\tr(\ad_s^2)=-\frac12\tr(\ad_s\circ\ad_s^*)-\frac12\tr(\ad_s^2)
			=-\frac14\tr((\ad_s+\ad_s^*)^2)\leq0
			\end{align}and hence $\ric(u,u)\leq0$ for any $u\in\G$. If $\G$ is unimodular then $\G$ is standard according \cite[Theorem 1]{dotti}. We suppose  that $\G$ is not unimodular and
				we choose $(u_1,\ldots,u_p)$ and orthonormal basis of $\n$ and $(v_1,\ldots,v_q)$ and orthonormal of $\af$.

			 We have $\tr(\ad_H)=|H|^2>0$,  $\tr(S(\ad_H)\ad_H)=\tr(S(\ad_H)^2)$ and since $\ad_H$ is a derivation, $\tr(B\circ\ad_H)=0$.
			Moreover,
			\[ \tr(\Ric\circ\ad_H)=\sum_{i}\langle [H,u_i],\Ric(u_i)\rangle+\sum_{j}\langle [H,v_j],\Ric(v_j)\rangle=c\sum_{i}\langle [H,u_i],u_i\rangle=c\tr(\ad_H). \]
			So, by taking $E=\ad_H$ in \eqref{eq4}, we get
			\begin{equation}\label{c} c=-\frac{\tr(S(\ad_H)^2)}{\tr(S(\ad_H))}. \end{equation}
			We identify $\n=[\G,\G]$ to $\R^p$ via $(u_1,\ldots,u_p)$ and we set $\mu=\br_{|\n\times\n}$. In this way,  $\mu$ can be
viewed as an element of $\mathscr{N}\subset V_p$. If $\mu\not=0$ then $\mu$ lies in a unique stratum $\mathscr{S}_{\be}$, $\be\in\mathscr{B}$, by \cite[Theorem 2.10, $(i)$]{lauret}. The argument used in the proof of Lauret's theorem to show that one can assume $\mu\in Y^{ss}_\be=\mathscr{S}_\be\cap W_\be$ can be used in our case, so we make the assumption. 

We apply \eqref{eq4} to the endomorphism $E$ given by $E_{|\af}=0$ and $E_{|\n}=\be+|\be|^2\mathrm{Id}_\n$. The right hand of the second equation in $\eqref{eq4}$ becomes
\begin{align*}
&\frac14\sum_{i,j}\langle E([u_i,u_j])-[E(u_i),u_j]-[u_i,E(u_j)],[u_i,u_j]\rangle
+\frac14\sum_{r,s}\langle E([v_r,v_s]),[v_r,v_s]\rangle\\
&+\frac12\sum_{i,r}\langle E([u_i,v_r]),[u_i,v_r]\rangle
-\frac12\sum_{r,i}\langle [v_r,E(u_i)],[v_r,u_i]\rangle,
\end{align*}which in turn equal to
\begin{align}\label{es}
&\frac12\langle\pi(\be+|\be|^2\mathrm{Id}_{\n})\mu,\mu\rangle+
\frac14\sum_{r,s}\langle(\be+|\be|^2\mathrm{Id}_{\n})[v_r,v_s],[v_r,v_s]\rangle
+\frac12\sum_{i,r}\langle(\be\ad_{v_r}-\ad_{v_r}\be)(u_i),\ad_{v_r}(u_i)\rangle.
\end{align}The first and second terms in \eqref{es} are  $\geq 0$ by \cite[Lemma 2.16, Lemma 2.17]{lauret}  and the last one equals $\frac12\langle[\be,\ad_{v_r}],\ad_{v_r}\rangle$ which is  $\geq 0$ by \cite[(10)]{lauret}.

Therefore, since $\Ric_{|\n}=c\mathrm{Id}_{\n}$ and $E_{|\af}=0$, we obtain from \eqref{eq4} and \eqref{c}
\begin{equation}\label{e}
-\frac{\tr(S(\ad_H)^2)}{\tr(S(\ad_H))}\tr(E)+\tr(S(\ad_H))E\geq0.
\end{equation}But $\tr(\be)=-1$ and so
\begin{equation}\label{ed} \tr(E^2)=|\be|^2(-1+p|\be|^2)=|\be|^2\tr(E). \end{equation}
	On the other hand,	we have
	\begin{equation}\label{ef} \tr(S(\ad_H)E)=\tr((\ad_H)_{|\n}(\be+|\be|^2\mathrm{Id_\n})=|\be|^2\tr(S(\ad_H)), \end{equation}by \cite[Lemma 2.15]{lauret}.	We use now \eqref{e}, \eqref{ed} and \eqref{ef} to obtain
	\begin{equation}\label{eg}\tr(S(\ad_H)^2)\tr(E^2)\leq (\tr(S(\ad_H)E)^2,
\end{equation}	a "backward" Cauchy-Schwartz inequality. This turns all inequalities mentioned after \eqref{e} into equalities, in particular the second term:
\[ \frac14\sum_{r,s}\langle(\be+|\be|^2\mathrm{Id}_{\n})[v_r,v_s],[v_r,v_s]\rangle=0. \]	We therefore get that $[\af,\af]=0$ since $\be+|\be|^2\mathrm{Id}_{\n}$ is positive definite by \cite[Lemma 2.17]{lauret}.

It only remains to consider the case   $\mu=0$. Here we argue in the same way but with $E$ chosen as $E_{|\n}=\mathrm{Id}_\n$ and $E_{|\af}=0$. It then follows from \eqref{eq4} that
\[ \frac14\sum_{r,s} |[v_r,v_s]|^2=-\frac{\tr(S(\ad_H)^2)}{\tr(S(\ad_H))}p+\tr(S(\ad_H))= 
\frac{\tr(S(\ad_H)^2)}{\tr(S(\ad_H))}\left( \frac{(\tr(S(\ad_H)))^2}{\tr(S(\ad_H)^2)}-p\right)\leq0      \]
and thus $[\n^\perp,\n^\perp]=0$. This concludes the proof of the theorem.
	\end{proof}

Let $(\G,\prs)$ be a Euclidean Lie algebra with harmonic curvature. According to Theorem \ref{prel},	 $\G=\G_1\oplus\ldots\oplus\G_r$, for any $i\in\{1,\ldots,r\}$, $\G_i$ is a subalgebra of $\G$ and  the equations \eqref{sk} and \eqref{3} hold. For any $u\in\G$, we denote by $u_i$ its component in $\G_i$. For any $i=1,\ldots,r$, we consider 
	\begin{equation}\label{h}\p_i=\sum_{j\not=i,k\not=i,j\not=k}[\G_j,\G_k]_i\esp \h_i=\p_i^\perp\cap\G_i.\end{equation}
	\begin{pr}\label{pr}$\h_i$ is a subalgebra which contains $[\rf_i,\rf_i]$ where $\rf_i$ is the radical of $\G_i$. Moreover, $[[\rf_i,\rf_i],\G]\subset\G_i$ and for any $u\in\h_i$ and any $v\in\G_i$
		\[ \ric_i(u,v)=\la_i\langle u,v\rangle. \]
		
	\end{pr}
	\begin{proof} By virtue of \eqref{3}, $u\in\h_i$ if and only if for any $j\not=i$, $[u,\G_j]\subset\G_i\oplus\G_j$. So if $u,v\in\h_i$ and $w\in\G_j$ with $j\not=i$ then
		\begin{align*}
		[[u,v],w]&=[u,[v,w]]-[v,[u,w]]\\
		&=[u,[v,w]_i]+[u,[v,w]_j]-[v,[u,w]_i]-[v,[u,w]_j]\subset\G_i\oplus\G_j,
		\end{align*}and hence $\h_i$ is a subalgebra.
		
		Let $\rf_i$ be the radical of $\G_i$ and  $\rho_i$ the representation of $\G_i$ given by \eqref{rho}.  Then $\rho_i(\rf_i)$ is a solvable subalgebra of $\mathrm{so}(\g_i^\perp)$ and hence it is abelian. So $\rho_i([\rf_i,\rf_i])=0$ and hence, for any $X\in[\rf_i,\rf_i]$, and any $j\not=i$, $[X,\G_j]\subset\G_i$ which shows that $
		[\rf_i,\rf_i]\subset\h_i$ and $[[\rf_i,\rf_i],\G]\subset\G_i$. The last assertion is a consequence of  \eqref{r}.
		\end{proof}	
		We can prove now our main theorem.
	\begin{theo}\label{main}Let $\G=\G_1\oplus\ldots\oplus\G_r$ be a Euclidean Lie algebra with harmonic curvature. If $\G_i$ is solvable  for any $i\in\{1,\ldots,e\}$ then $\G$ is Ricci-parallel. 
		\end{theo}\begin{proof} We will show that for any  $\{i,j,k\}$ which are mutually different, $\langle \G_i,[\G_j,\G_k]\rangle=0$. This is equivalent to $\p_i=\{0\}$ for any $i\in\{1,\ldots,r\}$, where $\p_i$ is given by \eqref{h}.

		 Suppose that $\la_1<\ldots<\la_r$. From \eqref{r}, we deduce that for any $u\in\G_r$,
	\[ \la_r|u|^2\leq\ric_r(u,u) \]and hence $\la_r\dim\G_r\leq \s_r$, where $\s_r$ is the scalar curvature of $\G_r$. But $\G_r$ is solvable and, according to \cite[Theorem 3.1]{milnor}, $\s_r<0$ or $\G_r$ is flat, thus $\la_r\leq0$. 
	
	Denote by $S$ the set of $j$ such that  $\G_j$ is not flat and let $i\in S$. Then, according to Proposition \ref{pr}, $[[\G_i,\G_i],\G]\subset\G_i$, $[\G_i,\G_i]\subset\h_i$ and hence for any $u\in\G_i$ and $v\in[\G_i,\G_i]$ 
	\[ \ric_i(u,v)=\la_i\langle u,v\rangle. \] By virtue of Theorem \ref{l}, $[\G_i,\G_i]^\perp\cap\G_i$ is abelian. Write
	\[ \G_i=[\G_i,\G_i]\oplus\uf_i\oplus\kf_i\esp [\G_i,\G_i]^\perp\cap\G_i=\uf_i\stackrel{\perp}\oplus\kf_i, \]where 
	\[ \kf_i=\{u\in [\G_i,\G_i]^\perp\cap\G_i, (\ad_u)_{|[\G_i,\G_i]}\;\mbox{is skew-symmetric} \}. \]
	Let us show that $\p_i\subset \kf_i $.	Indeed, for any $u_j\in\G_j$, $u_k\in\G_k$, and $j\not=k$ both different from $i$, and for any $d\in[\G_i,\G_i]$, by using the fact that $[[\G_i,\G_i],\G]\subset\G_i$ and \eqref{sk}, we get
	\[ 0=\langle[[u_j,u_k],d],d\rangle+\langle[[u_k,d],u_j],d\rangle+\langle[[d,u_j],u_k],d\rangle= \langle[[u_j,u_k]_i,d],d\rangle.\]
	On the other hand, by virtue of \eqref{r} and since $\p_i\subset \kf_i$, $\uf_i\subset\h_i$ and for any $u\in\uf_i$,
	\[ \ric_i(u,u)=\la_i\langle u,u\rangle. \]
	Moreover, for any  $u\in\kf_i$, $\ad_u$ is skew-symmetric in restriction to $[\G_i,\G_i]$. But $\uf_i\oplus\kf_i$ is abelian so $\ad_u$ is skew-symmetric in restriction to $\G_i$ and, by virtue of \eqref{k}, $\ric_i(u,u)=0$. So far we have shown that
	\[ \ric_i(u,u)=\begin{cases} \la_i|u|^2\;\mbox{if}\; u\in[\G_i,\G_i]\oplus \uf_i,\\0\;\mbox{if}\; u\in\kf_i \end{cases} \]which implies that $\s_i=\la_i(\dim\uf_i+\dim[\G_i,\G_i])$.
	  Then the  formula \eqref{s} can be written
	\[ \s=\la_1\dim\G_1+\ldots+\la_r\dim\G_r=\sum_{j\in S}\la_j(\dim[\G_j,\G_j]+\dim\uf_j). \]
	 So $\di \sum_{j\in S}\la_j\dim\kf_j+\sum_{j\notin S}\la_j\dim\G_j=0$ and, since $\la_1<\ldots<\la_r\leq0$,
	\[ \sum_{j\in S}\la_j\dim\kf_j=\sum_{j\notin S}\la_j\dim\G_j=0.  \]
	If $\la_r<0$ then for any $j$, $\G_j$ is not flat and $\kf_j=\{0\}$ which implies $\p_j=\{0\}$ and the result follows. If $\la_r=0$ then for any $j\leq r-1$, $\G_j$ is not flat and $\kf_j=\p_j=\{0\}$. We will have also  $\p_r=\{0\}$ by virtue of \eqref{3} which completes the proof.
\end{proof}

\begin{co} Let $\G$ be a  solvable Euclidean Lie algebra with harmonic curvature. Then $\G$ is Ricci-parallel.
	
	\end{co}

\begin{co}\label{co} Let $\G=\G_1\oplus\ldots\oplus\G_r$ be a Euclidean Lie algebra with harmonic curvature and
	not Ricci-parallel. Then there exists $i$ such that $\dim\G_i\geq3$.
\end{co}

\begin{co} Let $G/K$ be a homogeneous Riemannian manifold with harmonic curvature and $K$ is a maximal compact subgroup of $G$. Then $G/K$ is Ricci-parallel.
	
	\end{co}
	
	\begin{proof} Since $K$ is a maximal compact subgroup of $G$, there exists a solvable subgroup of $G$ which acts simply transitively on $G/K$ and hence $G/K$ is isometric to a solvable Lie group with a left invariant metric.
		\end{proof}

\section{Left invariant Riemannian metrics with harmonic curvature on  Lie groups of dimension $\leq6$ are Ricci-parallel}\label{section4}

In this section, we prove that any left invariant Riemannian metric with harmonic curvature on a   Lie group of dimension $\leq6$ is Ricci-parallel. Through-out this section $\G=\G_1\oplus\ldots\oplus \G_r$ is a Euclidean Lie algebra with harmonic curvature and we will not make any assumption on the order of the eigenvalues unless it is mentioned.

\begin{theo}\label{th} Let $\G=\G_1\oplus\G_2\oplus\G_3$ be a Euclidean   Lie algebra with harmonic curvature and $\dim\G_1=\dim\G_2=1$. Then $\G$ is Ricci-parallel.
	
\end{theo}

\begin{proof} Put $\G_1=\R u_1$, $\G_2=\R u_2$ with $|u_1|=|u_2|=1$. By virtue of \eqref{sk}, $\om=[u_1,u_2]\in\G_3$ and $\G$ is Ricci-parallel if and only if $\om=0$. Suppose that $\om\not=0$. As consequence of Proposition \ref{pr}, we get $\G_3=\h_3\oplus\R\om$ and for any $u\in\h_3$, $\ric_3(u,u)=\la_3\langle u,u\rangle$ and $\ric_3(u,\om)=0$.
	
Denote by $A_i:\G_3\too\G_3$, $v\mapsto [u_i,v]_3$. Both $A_1$ and $A_2$ are skew-symmetric.
	By virtue of \eqref{sk}, for any $v\in\G_3$, $[u_2,v]_2=[u_1,v]_1=0$ and hence
	\begin{align*}
	\;[\om,v]&=[u_1,[u_2,v]]-[u_2,[u_1,v]]\\
	&=[u_1,A_2(v)]-[u_2,A_1(v)]\\
	&=[A_1,A_2](v)+\langle [u_1,A_2(v)],u_2\rangle u_2-\langle [u_2,A_1(v)],u_1\rangle u_1.
	\end{align*}Thus
	\[ (\ad_\om)_{|\G_3}=[A_1,A_2]\esp A_1(\om)=A_2(\om)=0. \]
	On the other hand, for any $v,w\in\G_3$,
	\begin{align*}
	0&=\langle[u_1,[v,w]],u_2\rangle+\langle[v,[w,u_1]],u_2\rangle+\langle[w,[u_1,v]],u_2\rangle\\
	&=\langle[u_1,[v,w]],u_2\rangle+\langle[v,[w,u_1]_2],u_2\rangle+\langle[w,[u_1,v]_2],u_2\rangle\\
	&\stackrel{\eqref{sk},\eqref{3}}=-\frac{(\la_1-\la_3)^2}{(\la_1-\la_2)^2}
	\langle[v,w],[u_1,u_2]\rangle-\langle[w,u_1]_2,[v,u_2]\rangle-\langle[u_1,v]_2,
	[w,u_2]\rangle\\
	&=-\frac{(\la_1-\la_3)^2}{(\la_1-\la_2)^2}
	\langle[v,w],[u_1,u_2]\rangle
	\end{align*}since $[w,u_2]_2=[v,u_2]_2=0$. Thus $\om\in[\G_3,\G_3]^\perp\cap\G_3$.
	In conclusion, $(\ad_\om)_{|\G_3}$ is skew-symmetric and $\om\in[\G_3,\G_3]^\perp\cap\G_3$ and hence, by virtue of \eqref{k}, $\ric_3(\om,\om)=0$ and hence $\s_3=\la_3(\dim\G_3-1)$. By virtue of \eqref{s},
	\[ \s=\la_1+\la_2+\la_3\dim\G_3=\s_1+\s_2+\s_3=\la_3(\dim\G_3-1) \] and hence
	\begin{equation} \la_1+\la_2+\la_3=0. \end{equation}
	Let $(V_1,\ldots,V_p)$ and orthonormal basis of $\G_3$. By using \eqref{r}, we get
	\begin{align*}
	\la_1&=\ric(u_1,u_1)=\ric_1(u_1,u_1)-2\frac{(\la_3-\la_1)(\la_2-\la_1)}{(\la_3-\la_2)^2}\sum_{i}\langle[u_2,V_i],u_1\rangle^2\\
	&\stackrel{\eqref{3}}=-2\frac{(\la_3-\la_1)(\la_2-\la_1)(\la_2-\la_3)^4}{(\la_3-\la_2)^2(\la_1-\la_2)^4}\sum_{i}\langle V_i,\om\rangle^2\\
	&=2\frac{(\la_3-\la_1)(\la_2-\la_3)^2}{(\la_1-\la_2)^3}|\om|^2,\\
	\la_2&=\ric(u_2,u_2)=\ric_2(u_2,u_2)-2\frac{(\la_3-\la_2)(\la_1-\la_2)}{(\la_3-\la_1)^2}\sum_{i}\langle[u_1,V_i],u_2\rangle^2\\
	&\stackrel{\eqref{3}}=-2\frac{(\la_3-\la_2)(\la_1-\la_2)(\la_1-\la_3)^4}{(\la_3-\la_1)^2(\la_1-\la_2)^4}\sum_{i}\langle V_i,\om\rangle^2\\
	&=-2\frac{(\la_3-\la_2)(\la_1-\la_3)^2}{(\la_1-\la_2)^3}|\om|^2,\\
	\la_3|\om|^2&=\ric_3(\om,\om)-2\frac{(\la_1-\la_3)(\la_2-\la_3)}{(\la_1-\la_2)^2}\langle[u_1,u_2],\om\rangle^2\\
	&=-2\frac{(\la_1-\la_3)(\la_2-\la_3)}{(\la_1-\la_2)^2}|\om|^4.
	\end{align*}Put $\la_{ij}=\la_i-\la_j$.
	 So we must have
	\[\la_1+\la_2+\la_3=0, \la_1=-2\frac{\la_{23}^2\la_{13}}{\la_{12}^3}|\om|^2,\;\la_2=2\frac{\la_{13}^2\la_{23}}{\la_{12}^3}|\om|^2\esp\la_3=-2\frac{\la_{13}\la_{23}}{\la_{12}^2}|\om|^2. \]
	This is equivalent to
	\[ \la_1\la_{12}-\la_{23}\la_3=\la_2\la_{12}+\la_3\la_{13}=0\esp \la_3=-2\frac{\la_{13}\la_{23}}{\la_{12}^2}|\om|^2. \]
	This implies
	\[ \la_1^2-\la_1\la_2-\la_1(\la_1+\la_2)=0\esp-\la_2^2+\la_1\la_2+
	(\la_1+\la_2)\la_{2}=0. \]Thus $\la_1\la_2=0$ which is impossible. So $\om=0$ which completes the proof.
	\end{proof}

\begin{co} Any left invariant Riemannian metric with harmonic curvature on a Lie group $G$ of dimension $\leq5$ is Ricci-parallel.
	
\end{co}
\begin{proof} Let $\G$ be the Lie algebra of $G$. If the metric is non Ricci-parallel then $\Ric$ must have at least three distinct eigenvalues and at least one eigenvalue with multiplicity greater than 3 (See Corollary \ref{co}). This is not possible in dimensions 3 and 4. In dimension 5, the only possibility is $\G=\G_1\oplus\G_2\oplus\G_3$ where $\dim\G_3=3$ and $\dim\G_1=\dim\G_2=1$. This is not possible by virtue of Theorem \ref{th}.
	\end{proof}

 Let $\G$ be a 6-dimensional Euclidean Lie algebra with harmonic curvature and suppose that $\G$ is not Ricci-parallel. By virtue of Corollary \ref{co} and Theorem \ref{th}, we have two possibilities:
 \begin{enumerate}\item $\G=\G_1\oplus\G_2\oplus\G_3$ with $\dim\G_1=1$, $\dim\G_2=2$ and $\G_3$ is isomorphic either to $\mathrm{su}(2)$ or $\mathrm{sl}(2,\R)$.
 	\item $\G=\G_1\oplus\G_2\oplus\G_3\oplus\G_4$ with $\dim\G_1=\dim\G_2=\dim\G_3=1$ and $\G_4$ is isomorphic either to $\mathrm{su}(2)$ or $\mathrm{sl}(2,\R)$.
 	\end{enumerate}
We will see that both these  cases are impossible.

\begin{theo} Let $\G=\G_1\oplus\G_2\oplus\G_3$ be a Euclidean Lie algebra with harmonic curvature such that $\dim\G_1=1$, $\dim\G_2=2$ and $\dim\G_3=3$. Then $\G$ is Ricci-parallel.

	\end{theo}
	
	\begin{proof} We choose an orthonormal family $(X_1,X_2,X_3)$ such that $\G_1=\R X_1$, $\G_2=\mathrm{span}\{X_2,X_3\}$ and $[X_2,X_3]=aX_2$. The relation \eqref{sk} has many consequences. There exists $p\in\G_3$ and $r\in\R$ such that, for any $Y\in\G_3$, $$[X_1,X_2]_2=r X_3\esp [X_1,X_3]_2=-rX_2,\;
		[Y,X_2]_2=\langle p,Y\rangle X_3\esp [Y,X_3]_2=-\langle p,Y\rangle X_2.  $$
		Moreover, if we put  $\om_{1j}=[X_1,X_j]_3\in\G_3$ for $j=2,3$ and for $i=1,2,3$, $\rho_i:\G_3\too\G_3$, $Y\mapsto [X_i,Y]_3$, the $\rho_i$ are skew-symmetric. Finally, for any $Y\in\G$, $[X_1,Y]_1=0$.
		
		Now, for any  $Y\in\G_3$ and for $j=2,3$,
		\begin{align*}
		\;[[X_1,X_j],Y]&=[\om_{1j},Y]+[[X_1,X_j]_2,Y]\\
		&=[X_1,[X_j,Y]]-[X_j,[X_1,Y]]\\
		&=[X_1,\rho_j(Y)]+[X_1,[X_j,Y]_2]-[X_j,\rho_1(Y)]
		-[X_j,[X_1,Y]_2]\\
		&=\rho_1\circ\rho_j(Y)+[X_1,\rho_j(Y)]_2+[X_1,[X_j,Y]_2]-
		\rho_j\circ\rho_1(Y)-[X_j,\rho_1(Y)]_1-[X_j,\rho_1(Y)]_2-[X_j,[X_1,Y]_2].
		\end{align*}
			So we get
		\[ (\ad_{\om_{12}})_{|\G_3}=[\rho_1,\rho_2]-r\rho_3-\langle p,\bullet\rangle\om_{13},\;
		(\ad_{\om_{13}})_{|\G_3}=[\rho_1,\rho_3]+r\rho_2+\langle p,\bullet\rangle\om_{12},
		 \]and, for any $Y\in\G_3$,
		\begin{align*}
		[[X_1,X_j]_2,Y]_1+[[X_1,X_j]_2,Y]_2
		&=[X_1,\rho_j(Y)]_2+[X_1,[X_j,Y]_2]_2-[X_j,\rho_i(Y)]_2-[X_j,\rho_1(Y)]_1
		-[X_j,[X_1,Y]_2].
		\end{align*}So
		\[ \begin{cases}r\langle [X_3,Y],X_1\rangle+\langle [X_2,\rho_1(Y)],X_1\rangle=0,\\
		r\langle p,Y\rangle X_2=[X_1,\rho_2(Y)]_2+r\langle p,Y\rangle X_2+\langle p,\rho_1(Y)\rangle X_3-\langle [X_1,Y]_2,X_3\rangle [X_2,X_3],\\
		-r\langle [X_2,Y],X_1\rangle+\langle [X_3,\rho_1(Y)],X_1\rangle=0,\\
		r\langle p,Y\rangle X_3=[X_1,\rho_3(Y)]_2+r\langle p,Y\rangle X_3-\langle p,\rho_1(Y)\rangle X_2 -\langle[X_1,Y]_2,X_2\rangle[X_3,X_2].
		\end{cases} \]
		Thus
		\[ \begin{cases}
		\rho_1(\om_{12})=r\om_{13}\esp \rho_1(\om_{13})=-r\om_{12},\\
		0=[X_1,\rho_2(Y)]_2+\langle p,\rho_1(Y)\rangle X_3-\langle [X_1,Y]_2,X_3\rangle [X_2,X_3],\\
		0=[X_1,\rho_3(Y)]_2-\langle p,\rho_1(Y)\rangle X_2 -\langle[X_1,Y]_2,X_2\rangle[X_3,X_2],
		\end{cases} \]
		and finally,
		\begin{equation}\label{sum} \begin{cases}(\ad_{\om_{12}})_{|\G_3}=[\rho_1,\rho_2]-r\rho_3-\langle p,\bullet\rangle\om_{13},\\
		(\ad_{\om_{13}})_{|\G_3}=[\rho_1,\rho_3]+r\rho_2+\langle p,\bullet\rangle\om_{12},\\
		\rho_1(\om_{12})=r\om_{13}\esp \rho_1(\om_{13})=-r\om_{12},\\
		0=[X_1,\rho_2(Y)]_2+\langle p,\rho_1(Y)\rangle X_3-\langle [X_1,Y]_2,X_3\rangle [X_2,X_3],\\
		0=[X_1,\rho_3(Y)]_2-\langle p,\rho_1(Y)\rangle X_2 -\langle[X_1,Y]_2,X_2\rangle[X_3,X_2].
		\end{cases} \end{equation}
			On the other hand,
		\begin{align*}
		[[X_1,X_2],X_3]&=[\om_{12},X_3]=-\rho_3(\om_{12})-\langle p,\om_{12}\rangle X_2+\langle[\om_{12},X_3],X_1\rangle X_1,\\
		&=[X_1,[X_2,X_3]]-[X_2,[X_1,X_3]]=a[X_1,X_2]-[X_2,\om_{13}],\\
		[[X_1,X_3],X_2]&=[\om_{13},X_2]=-\rho_2(\om_{13})+\langle p,\om_{13}\rangle X_3+\langle[\om_{13},X_2],X_1\rangle X_1,\\
		&=[X_1,[X_3,X_2]]-[X_3,[X_1,X_2]]=-a[X_1,X_2]-[X_3,\om_{12}].
		\end{align*}	
		So
		\[\begin{cases} -\rho_3(\om_{12})-\langle p,\om_{12}\rangle X_2+\langle[\om_{12},X_3],X_1\rangle X_1=a\om_{12}+arX_3-\rho_2(\om_{13})+\langle p,\om_{13}\rangle X_3-\langle [X_2,\om_{13}],X_1\rangle X_1,\\
		-\rho_2(\om_{13})+\langle p,\om_{13}\rangle X_3+\langle[\om_{13},X_2],X_1\rangle X_1=
		-a\om_{12}-ar X_3-\rho_3(\om_{12})-\langle p,\om_{12}\rangle X_2-\langle[X_3,\om_{12}],X_1\rangle X_1,
		\end{cases} \]	
		and hence
		\[\langle p,\om_{12}\rangle=\langle p,\om_{12}\rangle=0,\quad ar=-\langle p,\om_{13}\rangle=-\langle p,\om_{12}\rangle,\quad -\rho_3(\om_{12})=a\om_{12}-\rho_2(\om_{13})\esp -\rho_2(\om_{13})=-a\om_{12} -\rho_3(\om_{12}). \]
		Suppose that $\G$ is not Ricci-parallel then $\dim\mathrm{span}\{\om_{12},\om_{13}\}\geq1$ and, according to Theorem \ref{main}, $\G_3$ is isomorphic either to $\mathrm{su}(2)$ or $\mathrm{sl}(2,\R)$.

		 Suppose that $\dim\mathrm{span}\{\om_{12},\om_{13}\}=1$. Then there exists $(\mu_1,\mu_2)\not=(0,0)$ such that $\mu_1\om_{12}+\mu_{2}\om_{13}=0$. The vector
		$X=\mu_2\om_{12}-\mu_1\om_{13}\not=0$ and, by virtue of \eqref{sum}, $\ad_X$ is skew-symmetric.
		The orthogonal of $X$ is the subalgebra $\h_3$ defined in \eqref{h} and since   $\ad_{X}$ is skew-symmetric it leaves invariant $\h_3$  and hence it is an ideal which is not possible since $\G_3$ is simple. 
		
		Suppose now that  $\dim\mathrm{span}\{\om_{12},\om_{13}\}=2$. We have
		\[ ar=\langle p,\om_{12}\rangle=\langle p,\om_{13}\rangle=0. \]
		If $r=0$ then we get from \eqref{sum} that $\rho_1(\om_{12})=\rho_1(\om_{13})=0$ and hence $\rho_1=0$ and then $[\om_{12},\om_{13}]=0$ which is impossible since $\G_3$ does not contains an abelian subalgebra.\\
		If $a=0$ and $r\not=0$. From \eqref{sum}, we have for any $Y\in\G_3$,
		\[ [X_1,\rho_2(Y)]_2+\langle p,\rho_1(Y)\rangle X_3=
		[X_1,\rho_3(Y)]_2-\langle p,\rho_1(Y)\rangle X_2 =0. \]
		But $p\in\{\om_{12},\om_{13}\}^\perp$ and hence $\rho_1(p)=0$.
		So
		\[ \langle [X_1,\rho_2(Y)]_2,X_2\rangle=\langle [X_1,\rho_3(Y)]_2,X_3\rangle=
		 \langle[X_1,\rho_2(Y)]_2,X_3\rangle=
		 \langle [X_1,\rho_3(Y)]_2,X_2\rangle=0. \]By using \eqref{3}, we get
		 that $\rho_2=\rho_3=0$. This implies that $[\om_{12},\om_{13}]=0$ which is impossible.
			\end{proof}

			\begin{pr}\label{prj} Let $\G=\G_1\oplus\G_2\oplus\G_3\oplus\G_4$ a Euclidean Lie algebra with harmonic curvature such that
				$\dim\G_1=\dim\G_2=\dim\G_3=1$, $\dim\G_4=3$ and $\la_1<\la_2<\la_3$. If $\G$ is not Ricci-parallel then there exists an orthonormal basis $(e_1,\ldots,e_6)$ of $\G$ with $e_i\in\G_i$ for $i=1,2,3$ and $(e_4,e_5,e_6)$ is a basis of $\G_4$ such that the non vanishing Lie brackets are
				\begin{equation} \label{j}\begin{cases}\;[e_4,e_5]=a e_6,\; [e_4,e_6]=b e_5\esp [e_5,e_6]=c e_4,\\\;[e_1,e_2]=\mu_1 e_4+\frac{\la_{12}^2}{\la_{23}^2}r e_3, [e_1,e_3]=\mu_2 e_5-\frac{\la_{13}^2}{\la_{23}^2}r e_2,\; [e_2,e_3]=\mu_3 e_6+r e_1,\\
				[e_1,e_4]=\al_1 e_5-\frac{\la_{14}^2}{\la_{12}^2}\mu_1 e_2,\;
				[e_1,e_5]=-\al_1 e_4-\frac{\la_{14}^2}{\la_{13}^2}\mu_2 e_3,\;
				\\
				[e_2,e_4]=\al_2 e_6+\frac{\la_{24}^2}{\la_{12}^2}\mu_1 e_1,\;
				[e_2,e_6]=-\al_2 e_4-\frac{\la_{24}^2}{\la_{23}^2}\mu_3 e_3,\\
				[e_3,e_5]=\al_3 e_6+\frac{\la_{34}^2}{\la_{13}^2}\mu_2 e_1,\;
				[e_3,e_6]=-\al_3 e_5+\frac{\la_{34}^2}{\la_{23}^2}\mu_3 e_2,
				\end{cases}\end{equation}where $a\not=0,b\not=0,c\not=0$, $\mu_i>0$ for $i=1,2,3$, $r,\al_1,\al_2,\al_3\in\R$ and $\la_{ij}=\la_i-\la_j$.

				\end{pr}

	\begin{proof} For $i=1,2,3$, put $\G_i=\R X_i$ with $|X_i|=1$  and $\om_{ij}=[X_i,X_j]_4$ for $1\leq i<j\leq 3$. By virtue of \eqref{sk}, we have
		\[ [X_1,X_2]=r_3X_3+\om_{12},\; [X_1,X_3]=r_2X_2+\om_{13}\esp [X_2,X_3]=r_1X_1+\om_{23}. \]
	 For any $X\in\G_1\oplus\G_2\oplus\G_3$, denote by $\rho(X):\G_4\too\G_4$, $Y\mapsto[X,Y]_4$ and for $i=1,2,3$ $\rho_i=\rho(X_i)$. For $Y\in\G_4$ and $1\leq i<j\leq 3$,
		\begin{align*}
		\;[[X_i,X_j],Y]&=[\om_{ij},Y]+\sum_{l\not=4}[[X_i,X_j]_l,Y],\\
		&=[X_i,[X_j,Y]]-[X_j,[X_i,Y]]\\
		&=[X_i,\rho_j(Y)]+\sum_{l\not=4}[X_i,[X_j,Y]_l]-[X_j,\rho_i(Y)]-
		\sum_{l\not=4}[X_j,[X_i,Y]_l].
		\end{align*}So
		\begin{equation}\label{eqij}
		[\om_{ij},Y]+\sum_{l\not=4}[[X_i,X_j]_l,Y]=[X_i,\rho_j(Y)]-[X_j,\rho_i(Y)]+
		\sum_{l\not=4}\left([X_i,[X_j,Y]_l]-
		[X_j,[X_i,Y]_l]\right).
		\end{equation}
		By taking the $\G_4$-component, we get
		\[ (\ad_{\om_{ij}})_{|\G_4}(Y)=[\rho_i,\rho_j](Y)
		-\sum_{l\not=4}\rho([X_i,X_j]_l)(Y)+\sum_{l\not=4}\left(
		[X_i,[X_j,Y]_l]_4-[X_j,[X_i,Y]_l]_4
		\right). \]
		Define $F_{ij}:\G_4\too\G_4$ by
		\[ F_{ij}(Y)=[X_i,[X_j,Y]_j]_4-[X_j,[X_i,Y]_i]_4+\sum_{l\not=4,l\not=i,l\not=j}\left(
		[X_i,[X_j,Y]_l]_4-[X_j,[X_i,Y]_l]_4
		\right). \]
		So we get
		\[ (\ad_{\om_{ij}})_{|\G_4}=[\rho_i,\rho_j]-\sum_{l\not=4}\rho([X_i,X_j]_l)+F_{ij}. \]Having \eqref{sk} and \eqref{3} in mind,  we have
		\begin{align*} F_{12}(Y)&=
		[X_1,[X_2,Y]_3]_4-[X_2,[X_1,Y]_3]_4=\langle [X_2,Y],X_3\rangle\om_{13}
		-\langle [X_1,Y],X_3\rangle\om_{23}\\&=-\frac{\la_{24}^2}{\la_{23}^2}\langle Y,\om_{23}\rangle\om_{13}+\frac{\la_{14}^2}{\la_{13}^2}\langle Y,\om_{13}\rangle\om_{23},\\
		F_{13}(Y)&=
		[X_1,[X_3,Y]_2]_4-[X_3,[X_1,Y]_2]_4=\langle [X_3,Y],X_2\rangle\om_{12}
		+\langle [X_1,Y],X_2\rangle\om_{23}\\&=\frac{\la_{34}^2}{\la_{23}^2}\langle Y,\om_{23}\rangle\om_{12}-\frac{\la_{14}^2}{\la_{12}^2}\langle Y,\om_{12}\rangle\om_{23},\\
		F_{23}(Y)&=
		[X_2,[X_3,Y]_1]_4-[X_3,[X_2,Y]_1]_4=-\langle [X_3,Y],X_1\rangle\om_{12}
		+\langle [X_2,Y],X_1\rangle\om_{13}\\&=-\frac{\la_{34}^2}{\la_{13}^2}\langle Y,\om_{13}\rangle\om_{12}+\frac{\la_{24}^2}{\la_{12}^2}\langle Y,\om_{12}\rangle\om_{13}.
		\end{align*} Thus
		\begin{equation}\label{ad} \begin{cases}(\ad_{\om_{12}})_{|\G_4}=[\rho_1,\rho_2]-r_3\rho_3-\frac{\la_{24}^2}{\la_{23}^2}\langle \bullet,\om_{23}\rangle\om_{13}+\frac{\la_{14}^2}{\la_{13}^2}\langle \bullet,\om_{13}\rangle\om_{23},\\
		(\ad_{\om_{13}})_{|\G_4}=[\rho_1,\rho_3]-r_2\rho_2+\frac{\la_{34}^2}{\la_{23}^2}\langle \bullet,\om_{23}\rangle\om_{12}-\frac{\la_{14}^2}{\la_{12}^2}\langle \bullet,\om_{12}\rangle\om_{23},\\
		(\ad_{\om_{23}})_{|\G_4}=[\rho_2,\rho_3]-r_1\rho_1-\frac{\la_{34}^2}{\la_{13}^2}\langle \bullet,\om_{13}\rangle\om_{12}+\frac{\la_{24}^2}{\la_{12}^2}\langle \bullet,\om_{12}\rangle\om_{13}.
		\end{cases} \end{equation}
		Now, for $(i,j)=(1,2)$ we identify the $\G_k$-component of \eqref{eqij} for $k=1,2,3$,
		\begin{equation}\label{12}\begin{cases}
		r_3\langle [X_3,Y],X_1\rangle=-\langle [X_2,\rho_1(Y)],X_1\rangle
		-\langle [X_2,[X_1,Y]_3],X_1\rangle,\\
		r_3\langle [X_3,Y],X_2\rangle=\langle [X_1,\rho_2(Y)],X_2\rangle+
		\langle [X_1,[X_2,Y]_3],X_2\rangle,\\
		0=\langle [X_1,\rho_2(Y)],X_3\rangle
		-\langle [X_2,\rho_1(Y)],X_3\rangle.\end{cases}
		\end{equation}We do the same for $(i,j)=(1,3)$
		\begin{equation}\label{13}\begin{cases}
		r_2\langle [X_2,Y],X_1\rangle=-\langle [X_3,\rho_1(Y)],X_1\rangle
		-\langle [X_3,[X_1,Y]_2],X_1\rangle,\\
		0=\langle [X_1,\rho_3(Y)],X_2\rangle-\langle [X_3,\rho_1(Y)],X_2\rangle,\\
		r_2\langle [X_2,Y],X_3\rangle=\langle [X_1,\rho_3(Y)],X_3\rangle
		+\langle [X_1,[X_3,Y]_2],X_3\rangle.\end{cases}
		\end{equation}Finally, we take $(i,j)=(2,3)$,
		\begin{equation}\label{23}\begin{cases}
		0=\langle [X_2,\rho_3(Y)],X_1\rangle-\langle [X_3,\rho_2(Y)],X_1\rangle
		,\\
		r_1\langle [X_1,Y],X_2\rangle=-\langle [X_3,\rho_2(Y)],X_2\rangle-\langle [X_3,[X_2,Y]_1],X_2\rangle,\\
		r_1\langle [X_1,Y],X_3\rangle=\langle [X_2,\rho_3(Y)],X_3\rangle
		+\langle [X_2,[X_3,Y]_1],X_3\rangle.\end{cases}
		\end{equation}
		By using \eqref{3}, we get that \eqref{12}-\eqref{23} are equivalent to
		\begin{equation*}\label{12b}\begin{cases}
		r_3\frac{\la_{34}^2}{\la_{13}^2}\langle Y,\om_{13}\rangle=
		\frac{\la_{24}^2}{\la_{12}^2}\langle Y,\rho_1(\om_{12})\rangle
		+r_1\frac{\la_{14}^2}{\la_{13}^2}\langle Y,\om_{13}\rangle,\\
		r_3\frac{\la_{34}^2}{\la_{23}^2}\langle Y,\om_{23}\rangle=
		\frac{\la_{14}^2}{\la_{12}^2}\langle Y,\rho_2(\om_{12})\rangle
		-r_2\frac{\la_{24}^2}{\la_{23}^2}\langle Y,\om_{23}\rangle,\\
		\langle\frac{\la_{14}^2}{\la_{13}^2}\rho_2(\om_{13})-\frac{\la_{24}^2}{\la_{23}^2}
		\rho_1(\om_{23}),Y\rangle=0,\end{cases}
		\end{equation*}
		\begin{equation*}\label{13b}\begin{cases}
		r_2\frac{\la_{24}^2}{\la_{12}^2}\langle Y,\om_{12}\rangle=
		\frac{\la_{34}^2}{\la_{13}^2}\langle Y,\rho_1(\om_{13})\rangle
		-r_1\frac{\la_{14}^2}{\la_{12}^2}\langle Y,\om_{12}\rangle,\\
		\frac{\la_{14}^2}{\la_{12}^2}\langle Y,\rho_3(\om_{12})\rangle
		+\frac{\la_{34}^2}{\la_{23}^2}\langle Y,\rho_1(\om_{23})\rangle=0,\\
		-r_2\frac{\la_{24}^2}{\la_{23}^2}\langle Y,\om_{23}\rangle 	=\frac{\la_{14}^2}{\la_{13}^2}\langle Y,\rho_3(\om_{13})\rangle
		+r_3\frac{\la_{34}^2}{\la_{23}^2}\langle Y,\om_{23}\rangle,
		\end{cases}\end{equation*}
		\begin{equation*}\label{23b}\begin{cases}
		-\frac{\la_{24}^2}{\la_{12}^2}\langle Y,\rho_3(\om_{12})\rangle
		+\frac{\la_{34}^2}{\la_{13}^2}\langle Y,\rho_2(\om_{13})\rangle=0,\\
		-r_1\frac{\la_{14}^2}{\la_{12}^2}\langle Y,\om_{12}\rangle 	=\frac{\la_{34}^2}{\la_{23}^2}\langle Y,\rho_2(\om_{23})\rangle
		+r_2\frac{\la_{24}^2}{\la_{12}^2}\langle Y,\om_{12}\rangle,\\
		-r_1\frac{\la_{14}^2}{\la_{13}^2}\langle Y,\om_{13}\rangle 	=\frac{\la_{24}^2}{\la_{23}^2}\langle Y,\rho_3(\om_{23})\rangle
		-r_3\frac{\la_{34}^2}{\la_{13}^2}\langle Y,\om_{13}\rangle.
		\end{cases}\end{equation*}
		These relations hold for any $Y\in\G_4$ and hence
		\begin{equation}\label{9}
		\begin{cases}\rho_1(\om_{12})=\frac{\la_{12}^2}{\la_{24}^2}\left(r_3\frac{\la_{34}^2}{\la_{13}^2}        -r_1\frac{\la_{14}^2}{\la_{13}^2}\right)\om_{13},\\
		\rho_1(\om_{13})=\frac{\la_{13}^2}{\la_{34}^2}\left( r_2\frac{\la_{24}^2}{\la_{12}^2}+r_1\frac{\la_{14}^2}{\la_{12}^2}   \right)
		\om_{12},
		\end{cases},\begin{cases}
		\rho_2(\om_{12})=\frac{\la_{12}^2}{\la_{14}^2}\left(r_3\frac{\la_{34}^2}{\la_{23}^2}+r_2\frac{\la_{24}^2}{\la_{23}^2}    \right)\om_{23},\\
		\rho_2(\om_{23})=-\frac{\la_{23}^2}{\la_{34}^2}\left( r_1\frac{\la_{14}^2}{\la_{12}^2}+ r_2\frac{\la_{24}^2}{\la_{12}^2}    \right)\om_{12},
		\end{cases},\begin{cases}
		\rho_3(\om_{13})=-\frac{\la_{13}^2}{\la_{14}^2}\left( r_2\frac{\la_{24}^2}{\la_{23}^2}+r_3\frac{\la_{34}^2}{\la_{23}^2}   \right)\om_{23},\\
		\rho_3(\om_{23})=\frac{\la_{23}^2}{\la_{24}^2}\left( r_3\frac{\la_{34}^2}{\la_{13}^2}-r_1\frac{\la_{14}^2}{\la_{13}^2}    \right)\om_{13},\\
		\end{cases}
		\end{equation}
		
		and\[ \begin{cases}
		{\la_{23}^2}{\la_{14}^2}\rho_2(\om_{13})-{\la_{13}^2}{\la_{24}^2}
		\rho_1(\om_{23})=0,\\
		{\la_{23}^2} {\la_{14}^2}\rho_3(\om_{12})
		+{\la_{12}^2}{\la_{34}^2}\rho_1(\om_{23})=0,\\
		-{\la_{13}^2}{\la_{24}^2}\rho_3(\om_{12})
		+{\la_{12}^2}{\la_{34}^2}\rho_2(\om_{13})=0.
		\end{cases} \]The determinant of this system in $(\rho_1(\om_{23}),\rho_2(\om_{13}),\rho_3(\om_{12}))$ in non zero and hence
		\begin{equation}\label{10} \rho_1(\om_{23})=\rho_2(\om_{13})=\rho_3(\om_{12})=0. \end{equation}
		From \eqref{ad}, \eqref{9} and \eqref{10}, $\mathrm{span}\{\om_{12},\om_{13},\om_{23}\}=\p_4$ is a subalgebra invariant by $\rho_i$ for $i=1,2,3$. Moreover, by virtue of Proposition \ref{h}, its orthogonal $\h_4$ is  a subalgebra. On the other hand, the endomorphisms $[\rho_1,\rho_2]-r_3\rho_3$, $[\rho_1,\rho_3]-r_2\rho_2$ and $[\rho_2,\rho_3]-r_1\rho_1$ are skew-symmetric and leave invariant $\p_4$ so they leave invariant $\h_4$. This implies, by virtue of \eqref{ad}, that $\h_4$ is an ideal of $\G_4$. But, since $\G$ is not Ricci-parallel, $\G_4$ is not solvable by virtue of Theorem \ref{main} and hence it is simple. But $\p_4\not=\{0\}$ otherwise $\G_4$ is Ricci-parallel. In conclusion $\h_4=0$ and $\G_4=\p_4$.
		
		Let us show now that $\{\om_{12},\om_{13},\om_{23}\}$  is orthonormal. This is true if $\rho_i\not=0$ for some $i$ by virtue of \eqref{9} and \eqref{10}. Suppose now that $\rho_1=\rho_2=\rho_3=0$. From \eqref{ad}, we get
		\[ [\om_{12},\om_{13}]=-\frac{\la_{24}^2}{\la_{23}^2}\langle \om_{13},\om_{23}\rangle\om_{13}+\frac{\la_{14}^2}{\la_{13}^2}\langle \om_{13},\om_{13}\rangle\om_{23}=-\frac{\la_{34}^2}{\la_{23}^2}\langle \om_{12},\om_{23}\rangle\om_{12}+\frac{\la_{14}^2}{\la_{12}^2}\langle \om_{12},\om_{12}\rangle\om_{23} \]
		This shows that $\langle \om_{13},\om_{23}\rangle=\langle \om_{12},\om_{23}\rangle=0$ and a similar argument gives $\langle \om_{13},\om_{12}\rangle=0$.
		
		In conclusion, we take $e_i=X_i$ for $i=1,2,3$ and  $(\om_{12},\om_{13},\om_{23})=(\mu_1e_4,\mu_2e_5,\mu_3e_6)$ to get the desired expression of the Lie brackets.
			\end{proof}

\begin{theo} Let $\G=\G_1\oplus\G_2\oplus\G_3\oplus\G_4$ a Euclidean Lie algebra with harmonic curvature such that
	$\dim\G_1=\dim\G_2=\dim\G_3=1$, $\dim\G_4=3$ and $\la_1<\la_2<\la_3$. Then $\G$ is Ricci-parallel.
	
	\end{theo}
	
	\begin{proof} Suppose that $\G$ is not Ricci-parallel. Then we can apply Proposition \ref{prj} and  there exists an orthonormal basis $\B=(e_1,\ldots,e_6)$ of $\G$ such that \eqref{j} holds. For any $u,v,w\in\G$, put 
		\[ J(u,v,w)=[u,[v,w]]+[v,[w,u]]+[w,[u,v]] \]and denote by $u_i$ the $i$-component of $u$ in the basis $\B$. A direct computation gives
		\[ \begin{cases} J(e_1,e_4,e_6)_4=-b\alpha_{{1}}-{\frac {\alpha_{{2}}{\lambda_{{14}}}^{2}\mu_{{1}}}{{
					\lambda_{{12}}}^{2}}}-\alpha_{{1}}c=0,\\
		J(e_1,e_5,e_6)_5=\alpha_{{1}}c-{\frac {\alpha_{{3}}{\lambda_{{14}}}^{2}\mu_{{2}}}{{
					\lambda_{{13}}}^{2}}}+b\alpha_{{1}}=0,\\
		J(e_2,e_4,e_5)_4=-a\alpha_{{2}}+{\frac {\alpha_{{1}}{\lambda_{{24}}}^{2}\mu_{{1}}}{{
					\lambda_{{12}}}^{2}}}+\alpha_{{2}}c=0,\\
		J(e_2,e_5,e_6)_6=\alpha_{{2}}c-{\frac {\alpha_{{3}}{\lambda_{{24}}}^{2}\mu_{{3}}}{{
					\lambda_{{23}}}^{2}}}-a\alpha_{{2}}=0,\\
		J(e_3,e_4,e_5)_5=-a\alpha_{{3}}+{\frac {\alpha_{{1}}{\lambda_{{34}}}^{2}\mu_{{2}}}{{
					\lambda_{{13}}}^{2}}}-\alpha_{{3}}b=0,\\
		J(e_3,e_4,e_6)_6=\alpha_{{3}}b+{\frac {\alpha_{{2}}{\lambda_{{34}}}^{2}\mu_{{3}}}{{
					\lambda_{{23}}}^{2}}}+a\alpha_{{3}}=0.\\
		\end{cases} \]If $E_i$ designs the $i$-equation in this system then $E_1+E_2=0$, $E_3-E_4=0$ and $E_5+E_6=0$ are equivalent to
		\[ \begin{cases}\di
		-{\frac {\alpha_{{2}}{\lambda_{{14}}}^{2}\mu_{{1}}}{{\lambda_{{12}}}
				^{2}}}-{\frac {\alpha_{{3}}{\lambda_{{14}}}^{2}\mu_{{2}}}{{\lambda_{{
							13}}}^{2}}}=0,\\\di
		{\frac {\alpha_{{1}}{\lambda_{{24}}}^{2}\mu_{{1}}}{{\lambda_{{12}}}^
				{2}}}+{\frac {\alpha_{{3}}{\lambda_{{24}}}^{2}\mu_{{3}}}{{\lambda_{{2
							3}}}^{2}}}=0,\\\di
		{\frac {\alpha_{{1}}{\lambda_{{34}}}^{2}\mu_{{2}}}{{\lambda_{{13}}}^
				{2}}}+{\frac {\alpha_{{2}}{\lambda_{{34}}}^{2}\mu_{{3}}}{{\lambda_{{2
							3}}}^{2}}}=0.
			\end{cases} \]
		This is equivalent to 
		\[ M\left(\begin{array}{c}\mu_1\\\mu_2\\\mu_3\end{array}   \right)
		=\left(\begin{array}{c}0\\0\\0\end{array}   \right)
		\esp M=\left( \begin {array}{ccc} -{\frac {\alpha_{{2}}{\lambda_{{14}}}^{2}
			}{{\lambda_{{12}}}^{2}}}&-{\frac {\alpha_{{3}}{\lambda_{{14}}}^{2}}{
			{\lambda_{{13}}}^{2}}}&0\\ \noalign{\medskip}{\frac {\alpha_{{1}}{
				\lambda_{{24}}}^{2}}{{\lambda_{{12}}}^{2}}}&0&{\frac {\alpha_{{3}}{
				\lambda_{{24}}}^{2}}{{\lambda_{{23}}}^{2}}}\\ \noalign{\medskip}0&{
		\frac {\alpha_{{1}}{\lambda_{{34}}}^{2}}{{\lambda_{{13}}}^{2}}}&{
		\frac {\alpha_{{2}}{\lambda_{{34}}}^{2}}{{\lambda_{{23}}}^{2}}}
	\end {array} \right). 
 \]Since $(\mu_1,\mu_2,\mu_3)\not=(0,0,0)$ we get
 \[ \det M=2\,{\frac {\alpha_{{2}}{\lambda_{{14}}}^{2}\alpha_{{3}}{\lambda_{{24
 				}}}^{2}\alpha_{{1}}{\lambda_{{34}}}^{2}}{{\lambda_{{12}}}^{2}{
 				\lambda_{{23}}}^{2}{\lambda_{{13}}}^{2}}}=0.
 	 \]We conclude that $\al_1=\al_2=\al_3=0$. 
 	 
 	 On the other hand,
 	 \[ \begin{cases}\di J(e_1,e_2,e_5)_1=-{\frac {\mu_{{2}}r \left( \lambda_{{12}}\lambda_{{34}}-
 	 		\lambda_{{14}}\lambda_{{23}} \right)  \left( \lambda_{{12}}\lambda_
 	 		{{34}}+\lambda_{{14}}\lambda_{{23}} \right) }{{\lambda_{{23}}}^{2}
 	 		{\lambda_{{13}}}^{2}}}=0\\\di
 	 J(e_1,e_2,e_6)_2=-{\frac {\mu_{{3}}r \left( \lambda_{{12}}\lambda_{{34}}-
 	 		\lambda_{{13}}\lambda_{{24}} \right)  \left( \lambda_{{12}}\lambda_
 	 		{{34}}+\lambda_{{13}}\lambda_{{24}} \right) }{{\lambda_{{23}}}^{4}
 	 	}}=0\\\di
 	 J(e_1,e_3,e_4)_1={\frac {\mu_{{1}}r \left( \lambda_{{13}}\lambda_{{24}}-\lambda
 	 		_{{14}}\lambda_{{23}} \right)  \left( \lambda_{{13}}\lambda_{{24}}
 	 		+\lambda_{{14}}\lambda_{{23}} \right) }{{\lambda_{{12}}}^{2}{
 	 			\lambda_{{23}}}^{2}}}=0.
 	 \end{cases} \]If $r\not=0$ this is equivalent to
 	 \[ \left( \lambda_{{12}}\lambda_{{34}}-
 	 \lambda_{{14}}\lambda_{{23}} \right)  \left( \lambda_{{12}}\lambda_
 	 {{34}}+\lambda_{{14}}\lambda_{{23}} \right)=
 	 \left( \lambda_{{12}}\lambda_{{34}}-
 	 \lambda_{{13}}\lambda_{{24}} \right)  \left( \lambda_{{12}}\lambda_
 	 {{34}}+\lambda_{{13}}\lambda_{{24}} \right)=0.  \]But
 	 \[\begin{cases}\left( \lambda_{{12}}\lambda_{{34}}-
 	 \lambda_{{13}}\lambda_{{24}} \right)=(\la_2-\la_3)(\la_4-\la_1),\\
 	 \left( \lambda_{{12}}\lambda_{{34}}+
 	 \lambda_{{14}}\lambda_{{23}} \right)=(\la_1-\la_3)(\la_2-\la_4),\\
 	 \left( \lambda_{{12}}\lambda_{{34}}-
 	 \lambda_{{14}}\lambda_{{23}} \right)+
 	 \left( \lambda_{{12}}\lambda_{{34}}+
 	 \lambda_{{13}}\lambda_{{24}} \right)=3(\la_1-\la_2)(\la_3-\la_4),
 	 \end{cases}
 	   \]which shows that this situation is impossible and hence $r=0$.
 	   
 	  So far we have shown that $r=\al_1=\al_2=\al_3=0$. Now,
 	   \[ J(e_1,e_2,e_5)_6={\frac {{\lambda_{{14}}}^{2}\mu_{{2}}\mu_{{3}}}{{\lambda_{{13}}}^{2}
 	   	}}-\mu_{{1}}a
 	   	,\; J(e_1,e_2,e_6)_5=-{\frac {{\lambda_{{24}}}^{2}\mu_{{3}}\mu_{{2}}}{{\lambda_{{23}}}^{2
 	   			}}}-\mu_{{1}}b,\;
 	   			 J(e_1,e_3,e_6)_4={\frac {{\lambda_{{34}}}^{2}\mu_{{3}}\mu_{{1}}}{{\lambda_{{23}}}^{2}
 	   			 	}}-\mu_{{2}}c
 	   			 	\]and  hence
 	 \[ a={\frac {{\lambda_{{14}}}^{2}\mu_{{2}}\mu_{{3}}}{{\lambda_{{13}}}^{2}
 	 		\mu_{{1}}}},\;b=-{\frac {{\lambda_{{24}}}^{2}\mu_{{3}}\mu_{{2}}}{{\lambda_{{23}}}^{2
 	 		}\mu_{{1}}}}\esp c={\frac {{\lambda_{{34}}}^{2}\mu_{{3}}\mu_{{1}}}{{\lambda_{{23}}}^{2}
 	 		\mu_{{2}}}}.
 	   \]With all the relations established above, the vanishing of the Jacobi identity is equivalent to 
 	   \[ {\lambda_{{12}}}^{2}{\mu_{{2}}}^{2}-{\lambda_{{13}}}^{2}{\mu_{{1}}}
 	   ^{2}={\lambda_{{12}}}^{2}{\mu_{{3}}}^{2}-{\lambda_{{23}}}^{2}{\mu_{{
 	   			1}}}^{2}={\lambda_{{13}}}^{2}{\mu_{{3}}}^{2}-{\lambda_{{23}}}^{2}{
 	   	\mu_{{2}}}^{2}=0,
 	    \]which is equivalent to
 	    \[ \mu_2=\frac{\la_{13}\mu_1}{\la_{12}}\esp \mu_{3}=\frac{\la_{23}\mu_1}{\la_{12}}. \]
		We replace in $a,b,c$ and we get
		\[ a=\frac{\la_{14}^2\la_{23}\mu_1}{\la_{12}^2\la_{13}},\;
		b=-\frac{\la_{24}^2\la_{13}\mu_1}{\la_{12}^2\la_{23}}\esp c=
		\frac{\la_{34}^2\mu_1}{\la_{23}\la_{13}}. \]
		Now, by using Maple we get
		\[ \begin{cases}\di\na_{e_1}(\Ric)(e_2,e_4)-
		\na_{e_2}(\Ric)(e_1,e_4)=4\,{\frac {{\mu_{{1}}}^{3} \left( \lambda_{{1}}-\lambda_{{4}}
				\right)  \left( \lambda_{{2}}-\lambda_{{4}} \right)  \left( {\lambda_
					{{1}}}^{2}+ \left( -\lambda_{{2}}-\lambda_{{4}} \right) \lambda_{{1}}+
				{\lambda_{{2}}}^{2}-\lambda_{{2}}\lambda_{{4}}+{\lambda_{{4}}}^{2}
				\right) }{ \left( \lambda_{{1}}-\lambda_{{2}} \right) ^{4}}},\\
	\di	\na_{e_1}(\Ric)(e_3,e_5)-
		\na_{e_3}(\Ric)(e_1,e_5)=4\,{\frac {{\mu_{{1}}}^{3} \left( \lambda_{{1}}-\lambda_{{4}}
				\right)  \left( \lambda_{{3}}-\lambda_{{4}} \right)  \left( {\lambda_
					{{1}}}^{2}+ \left( -\lambda_{{3}}-\lambda_{{4}} \right) \lambda_{{1}}+
				{\lambda_{{3}}}^{2}-\lambda_{{3}}\lambda_{{4}}+{\lambda_{{4}}}^{2}
				\right) }{ \left( \lambda_{{1}}-\lambda_{{3}} \right)  \left( \lambda
				_{{1}}-\lambda_{{2}} \right) ^{3}}},\\
	\di	\na_{e_2}(\Ric)(e_3,e_6)-
		\na_{e_3}(\Ric)(e_2,e_6)=4\,{\frac { \left( \lambda_{{2}}-\lambda_{{4}} \right) {\mu_{{1}}}^{3
				} \left( {\lambda_{{2}}}^{2}+ \left( -\lambda_{{3}}-\lambda_{{4}}
				\right) \lambda_{{2}}+{\lambda_{{3}}}^{2}-\lambda_{{3}}\lambda_{{4}}+
				{\lambda_{{4}}}^{2} \right)  \left( \lambda_{{3}}-\lambda_{{4}}
				\right) }{ \left( \lambda_{{2}}-\lambda_{{3}} \right)  \left( \lambda
				_{{1}}-\lambda_{{2}} \right) ^{3}}}.
			\end{cases} \]
			So we must have
			\[\begin{cases} q_1:={\lambda_{{1}}}^{2}+ \left( -\lambda_{{2}}-\lambda_{{4}} \right) 
			\lambda_{{1}}+{\lambda_{{2}}}^{2}-\lambda_{{2}}\lambda_{{4}}+{\lambda_
				{{4}}}^{2}=0,\\
			q_2:={\lambda_{{1}}}^{2}+ \left( -\lambda_{{3}}-\lambda_{{4}} \right) 
			\lambda_{{1}}+{\lambda_{{3}}}^{2}-\lambda_{{3}}\lambda_{{4}}+{\lambda_
				{{4}}}^{2}=0,\\
			q_3:={\lambda_{{2}}}^{2}+ \left( -\lambda_{{3}}-\lambda_{{4}} \right) 
			\lambda_{{2}}+{\lambda_{{3}}}^{2}-\lambda_{{3}}\lambda_{{4}}+{\lambda_
				{{4}}}^{2}=0.
			\end{cases}
		 \]This is equivalent
		 \[ \begin{cases}q_1-q_2=- \left( \lambda_{{2}}-\lambda_{{3}} \right)  \left( -\lambda_{{2}}+
		 \lambda_{{1}}+\lambda_{{4}}-\lambda_{{3}} \right) =0,\\
		 q_1-q_3=\left( \lambda_{{1}}-\lambda_{{3}} \right)  \left( \lambda_{{1}}-
		 \lambda_{{2}}-\lambda_{{4}}+\lambda_{{3}} \right)=0,\\ 
		 q_3=0
		 \end{cases} \]
		which is equivalent to
		\[ \begin{cases}p_1=  \left( -\lambda_{{2}}+
		\lambda_{{1}}+\lambda_{{4}}-\lambda_{{3}} \right) =0,\\
		p_2=  \left( \lambda_{{1}}-
		\lambda_{{2}}-\lambda_{{4}}+\lambda_{{3}} \right)=0,\\ 
		q_3=0.
		\end{cases} \]This implies that $p_1+p_2=2(\la_1-\la_2)=0$ which is impossible. This completes the proof.
			\end{proof}
			Recall that a left-invariant symmetric tensor field $A$ on $\G$ is called  essential Codazzi if it satisfies the Codazzi equation, it is non parallel
			 and none of its  eigenspace subalgebras $\G_k$ is an ideal.
			The proof of the precedent result gives the following:
\begin{pr} We consider the 6-dimensional Euclidean Lie algebra with an orthonormal basis $(e_1,\ldots,e_6)$ in which  the non-vanishing Lie brackets are given by
	\[ \begin{cases}[e_1,e_2]=\mu_1e_4,[e_1,e_3]={\frac { \left( \lambda_{{1}}-\lambda_{{3}} \right) \mu_{{1}}}{\lambda
			_{{1}}-\lambda_{{2}}}}e_5,[e_1,e_4]=-{\frac { \left( \lambda_{{1}}-
			\lambda_{{4}} \right) ^{2}\mu_{{1}}}{ \left( \lambda_{{1}}-\lambda_{{2
				}} \right) ^{2}}}e_2, [e_1,e_5]=-{\frac { \left( \lambda_{{1}}-
				\lambda_{{4}} \right) ^{2}\mu_{{1}}}{ \left( \lambda_{{1}}-\lambda_{{3
					}} \right)  \left( \lambda_{{1}}-\lambda_{{2}} \right) }}e_3,\\
			[e_2,e_3]={\frac { \left( \lambda_{{2}}
					-\lambda_{{3}} \right) \mu_{{1}}}{\lambda_{{1}}-\lambda_{{2}}}}e_6, [e_2,e_4]={\frac { \left( \lambda_{{2}}-\lambda_{
						{4}} \right) ^{2}\mu_{{1}}}{ \left( \lambda_{{1}}-\lambda_{{2}}
					\right) ^{2}}}e_1,[e_2,e_6]=-{\frac { \left( \lambda_{{2}}-
					\lambda_{{4}} \right) ^{2}\mu_{{1}}}{ \left( \lambda_{{2}}-\lambda_{{3
						}} \right)  \left( \lambda_{{1}}-\lambda_{{2}} \right) }}e_3,[e_3,e_5]={\frac { \left( \lambda_{{3}}-\lambda_{
						{4}} \right) ^{2}\mu_{{1}}}{ \left( \lambda_{{1}}-\lambda_{{3}}
					\right)  \left( \lambda_{{1}}-\lambda_{{2}} \right) }}e_1,\\
			[e_3,e_6]={\frac { \left( \lambda_{{3}}-\lambda
					_{{4}} \right) ^{2}\mu_{{1}}}{ \left( \lambda_{{2}}-\lambda_{{3}}
					\right)  \left( \lambda_{{1}}-\lambda_{{2}} \right) }}e_2,
			[e_4,e_5]={\frac { \left( \lambda_{{1}}
					-\lambda_{{4}} \right) ^{2}\mu_{{1}} \left( \lambda_{{2}}-\lambda_{{3}
					} \right) }{ \left( \lambda_{{1}}-\lambda_{{3}} \right)  \left( 
					\lambda_{{1}}-\lambda_{{2}} \right) ^{2}}}e_6,[e_4,e_6]=-\frac{\la_{24}^2\la_{13}\mu_1}{\la_{12}^2\la_{23}}e_5,\;
			[e_5,e_6]=
			{\frac { \left( \lambda_{{3}}-
					\lambda_{{4}} \right) ^{2}\mu_{{1}}}{ \left( \lambda_{{1}}-\lambda_{{3
						}} \right)  \left( \lambda_{{2}}-\lambda_{{3}} \right) }}e_4.
				\end{cases} \]Then the operator $A=\mathrm{Diag}(\la_1,\la_2,\la_3,\la_4,\la_4,\la_4)$ is a symmetric essential Codazzi tensor.  This Lie algebra is compact (its Killing form is negative definite) and hence it is isomorphic to $\mathrm{su}(2)\oplus\mathrm{su}(2)$.
	
	\end{pr}

\end{document}